\documentclass{amsart}

\usepackage{latexsym}
\usepackage{amssymb}
\usepackage{graphicx}
\usepackage{subcaption} 

\usepackage{hyperref}
\hypersetup{pdfborder = {0 0 0}}
\newtheorem{THM}{Theorem}
\newtheorem{Lemma}[THM]{Lemma}

\theoremstyle{remark}

\newcommand{\bR}{\mathbb{R}}
\newcommand{\bH}{\mathbb{H}}
\newcommand{\be}{\begin{equation}}
\newcommand{\ee}{\end{equation}}
\newcommand{\bS}{\mathbb{S}}
\newcommand{\mS}{\mathbb{S}}

\newcommand{\mC}{\mathbb{C}}

\newcommand{\ra}{\rightarrow}

\newcommand{\ip}[2]{\left\langle#1,#2\right\rangle}
\newcommand{\lp}{\left(}
\newcommand{\rp}{\right)}
\newcommand{\lb}{\left[}
\newcommand{\rb}{\right]}
\newcommand{\lc}{\left\{}
\newcommand{\rc}{\right\}}
\newcommand{\lab}{\left|}
\newcommand{\rab}{\right|}
\newcommand{\Lap}{\Delta}

\newcommand{\E}{\mathbb{E}}
\newcommand{\Prob}{\mathbb{P}}

\DeclareMathOperator{\area}{area}

\renewcommand{\d}[1]{\,d#1}

\newcommand{\Ind}[1]{\boldsymbol 1_{\{#1\}}}

\newcommand{\eps}{\varepsilon}

\newcommand{\SI}{\mathbb{S}_{\infty}}
\DeclareMathOperator{\supp}{supp}

\newcommand{\vn}{\eta}

\begin{document}

\title{Geometric and Martin Boundaries of a Cartan-Hadamard surface}
\author{Robert W.\ Neel}
\address{Department of Mathematics, Lehigh University, Bethlehem, PA, USA}
\email{robert.neel@lehigh.edu}
\begin{abstract}
We give a general criterion for the Dirichlet problem at infinity (DPI) on a Cartan-Hadamard surface to be solvable, which we primarily use to give the best possible upper radial radial curvature bound for solvability of the DPI, but which is also flexible enough to accommodate flats. In particular, any (upper) radial curvature bound which implies transience also implies solvability of the DPI, which is perhaps surprising. Taking advantage of the structure provided by uniformization, we show that solvability of the DPI implies there is a natural continuous surjection of the Martin boundary onto the geometric boundary at infinity. Finally, we give matched upper and lower radial curvature bounds that imply the natural identification of the geometric and Martin boundaries (for Cartan-Hadamard surfaces) that are more generous than the bounds that are known in arbitrary dimension.
\end{abstract}

\maketitle

\section{Introduction}

The space of positive harmonic functions on a Riemannian manifold and its relationship to geodesic geometry has been studied for many years, by both analytic and stochastic methods. One natural context in which to study this question is that of Cartan-Hadamard manifolds satisfying radial curvature bounds. In particular, if $M$ is a Cartan-Hadamard manifold of dimension $n$, any point $p$ is a pole (meaning the exponential map at $p$ is a diffeomorphism), so that there are polar coordinates $(r,\theta)\in [0,\infty)\times \mS^{n-1}$ on $M$ based at $p$. Then we are interested in curvature bounds that depend on $r$.

\subsection{Basic definitions} A Cartan-Hadamard manifold has a geometric boundary at infinity $\SI(M)$, which is topologically a sphere of dimension $n-1$, given abstractly as equivalence classes of rays, and there is a natural topology, the cone topology, on $M\cup \SI(M)$ making it compact. Concretely, for any pole $p$ and associated polar coordinates $(r,\theta)$, $\SI(M)$ can be identified with the unit sphere in $T_pM$ so that $\theta$ gives coordinates on $\SI(M)$, and for any other pole $p'$ and polar coordinates $(r',\theta')$, the induced map from $\theta$ to $\theta'$ is a homeomorphism. Moreover, a sequence of points $(r_i,\theta_i)$ in $M$ converges to a point $\hat{\theta}_0\in \SI(M)$ if and only if $r_i\ra\infty$ and $\theta_i\ra \hat{\theta}_0$.

Let $g$ be a continuous function on $\SI(M)$. Then the Dirichlet problem at infinity (abbreviated DPI in what follows) on $M$ for boundary data $g$ is the problem of finding a function $u$ which is continuous on $M\cup \SI(M)$, harmonic on $M$, and such that $u|_{\SI(M)} = g$. We say the DPI is solvable on $M$ if the DPI has a unique solution for every choice of continuous boundary data $g$. One of our primary interests is curvature conditions on $M$ which imply solvability of the DPI.

More generally, we can consider the Martin boundary of $M$, which we denote $\partial M$, and the associated Martin compactification  $M\cup \partial M$. (In general, one can see \cite{SchoenYau} for the basic definitions, but soon we will specialize to surfaces, for which a simpler description of the Martin boundary can be given. Intuitively, the Martin boundary encodes the space of all positive harmonic functions on $M$.) Then one might also hope for curvature conditions under which $\SI(M)$ and $\partial M$ are naturally homeomorphic (meaning that if one considers the geometric compactification of $M$ and the Martin compactification of $M$, then the identity map on $M$ extends to a homeomorphism of these two compactifications), which is a stronger result than solvability of the DPI.

Before briefly describing the previous results on the DPI and the identification of the Martin boundary under radial curvature bounds, we note that there is an enormous literature on potential theory, in a variety of contexts, via both analytic and stochastic methods. In terms of other approaches to Cartan-Hadamard manifolds, we mention the following (incomplete list of) examples of criteria other than the type of radial curvature bounds with which we work-- convexity at infinity (see \cite{Choi}), visibility at the boundary and Gromov hyperbolicity (see \cite{Kifer}), and (ratios of) radial curvature bounds sufficient to also treat $p$-Laplacians (see \cite{HAndV}).

\subsection{Previous work in arbitrary dimensions}\label{Sect:Previous}
It was famously shown in \cite{AS} that, if the sectional curvatures of $M$ are pinched between negative constants, then $\SI(M)$ and $\partial M$ are naturally homeomorphic. To go beyond this, one can consider radial curvature bounds in which the the upper bound is allowed to decay to 0 and/or the lower bound is allowed to blow up to $-\infty$. In this context, the best known radial curvature bounds for solvability of the DPI are due to Hsu \cite{EltonBook} (see also \cite{EltonArticle}). By studying the angular behavior of Brownian motion, he proved that if $M$ is a Cartan-Hadamard manifold (of any dimension), and for some pole $p$ one of the followings set of curvature bounds holds:
\begin{itemize}
\item for constants $0<\lambda<2$ and $C>0$,
\[
-Ce^{\lambda r} \leq \mathrm{Ric} \quad\text{and}\quad K\leq -1 ,
\]
or
\item for positive constants $\alpha>2$, $\beta<\alpha-2$,
\[
-r^{2\beta} \leq \mathrm{Ric} \quad\text{and}\quad K\leq -\frac{\alpha(\alpha-1)}{r^2}
\]
outside of a compact,
\end{itemize}
then the DPI is solvable on $M$. Here $K$ is the sectional curvature as a function of $r$, $\theta$, and a plane in $T_{(r,\theta)}M$, and $\mathrm{Ric}$ is the Ricci curvature as a function of $r$, $\theta$, and a unit vector in $T_{(r,\theta)}M$. (Note that in the first set of conditions, the upper bound of $-1$ is just a choice of normalization; rescaling by a factor of $a>0$ gives an upper bound of $-a^2$ for the sectional curvature and corresponding lower bound of $-Ce^{\lambda a r}$, perhaps for a different choice of $C>0$ and $\lambda\in(0,2)$, for the Ricci curvature.) The proof of these results amounts to showing that a certain series has a finite sum, where the size of the terms coming from a given (higher-dimensional) annulus (in $r$) is controlled by the upper curvature bound, while the number of terms coming from a given annulus is (primarily) controlled by the lower curvature bound. Thus any pair of bounds that adequately controls this series should be sufficient to obtain the result, and the two cases above are just particularly natural pairs (indeed, see Proposition 6.1.2 of \cite{EltonBook} and the surrounding discussion). We encounter a similar phenomenon in our Theorem \ref{THM:MartinId}, as exhibited in the argument in Section \ref{Sect:MartinId}.

Recently, Ji \cite{RanJi} gave an analytic proof of (essentially) Hsu's results, and furthermore, gave similar, but more restrictive, curvature bounds under which the geometric and Martin boundaries can be shown to naturally homeomorphic. In particular, he showed that, for such an $M$, if either of the following sets of curvature bounds holds (relative to some pole):
\begin{itemize}
\item for constants $0<\lambda<\frac{2}{3}$ and $C>0$,
\[
-Ce^{\lambda r} \leq K\leq -1 ,
\]
or
\item for positive constants $\alpha>2$ and $\beta<\frac{\alpha-4}{3}$,
\[
-r^{2\beta} \leq \mathrm{Ric} \quad\text{and}\quad K\leq -\frac{\alpha(\alpha-1)}{r^2}
\]
outside of a compact,
\end{itemize}
then the geometric and Martin boundaries of $M$ are naturally homeomorphic.

\subsection{Previous work for surfaces}
The case of surfaces (that is, $n=2$ in the above) is substantially different from the higher dimensional case. In particular, Kendall and Hsu \cite{HsuKendall} showed that on a Cartan-Hadamard surface, the upper curvature bound
\[
K \leq -\frac{C}{r^2} \quad \text{for $r>R$}
\]
(relative to some pole, where $C$ and $R$ are positive constants) is sufficient for solvability of the DPI, with no lower curvature bound required. (In this case, $K$ reduces to a scalar function of $(r,\theta)$, namely the Gauss curvature.) That a lower curvature bound is needed in dimension 3 and higher is known from the work of Ancona \cite{AnconaCounterEx}, where a variety of examples of Cartan-Hadamard manifolds of dimension 3 with all sectional curvatures bounded above by $-1$ but on which the DPI is not solvable are constructed (the ways in which the DPI fails to be solvable and its relationship to the behavior of Brownian motion and the existence of bounded harmonic functions is an interesting aspect, which is explored in \cite{AntonMarcDPI}). Kendall and Hsu note that their method does not extend to a better upper bound. In \cite{MyDPI}, the author developed a different, though still probabilistic, approach to show that the DPI on a Cartan-Hadamard surface is solvable under the classical condition for transience of Milnor \cite{Milnor}, that
\[
K \leq -\frac{1+\eps}{r^2 \log r} \quad\text{for $r>R$}
\]
for some $\eps>0$ and some $R>1$.

\subsection{Summary of results}
The purpose of the present work is to improve this result for solvability of the DPI on Cartan-Hadamard surfaces, in particular, to give the best possible radial curvature bound, and to give curvature bounds under which the Martin boundary and geometric boundary on such a surface are naturally homeomorphic which are more generous than the bounds for higher dimensions given above. We now describe these results in more detail.

Recall that in polar coordinates around a pole $p$, the metric on $M$ can be written as 
\[
ds^2 = dr^2 + J^2(r,\theta) \d\theta^2
\]
where the notation is chosen to reflect the fact that, for each $\theta$, $J(r,\theta)$, as a function of $r$, is the length of the Jacobi field perpendicular to the radial geodesic ray $r\mapsto (r,\theta)$. Then our fundamental criterion for solvability of the DPI in terms of $J(r,\theta)$ is the following (see Section \ref{Sect:InJ}).

\begin{THM}\label{THM:MainDPI}
Let $M$ be a Cartan-Hadamard surface with polar coordinates around a point $p$. Suppose that for any $\theta_0$, any $\beta\in(0,\pi)$, and any $\eps>0$, there exists $\rho>0$ such that
\begin{equation}\label{Eqn:MainHypo}
\int_{\{ r>\alpha , |\theta- \theta_0|<\beta\}} \frac{1}{J^2(r,\theta)}\,d(\area)
=  \int_{\{ r>\alpha , |\theta- \theta_0|<\beta\}} \frac{1}{J(r,\theta)}\,dr\,d\theta < \eps
\end{equation}
whenever $\alpha>\rho$. Then the DPI on $M$ is solvable. 
\end{THM}

It is interesting to observe that a sufficient condition for transience on a surface with a pole, going back to Doyle \cite{Doyle} (but see Theorem 12.1 of \cite{Grigoryan} for a rigorous proof) is as follows. If the (Lebesgue) measure of the set
\[
\lc \theta\in\mS^1 : \int^{\infty} \frac{1}{J(r,\theta)}\, dr <\infty  \rc
\]
is positive, then $M$ is transient. Equation \eqref{Eqn:MainHypo} implies that $\int^{\infty} 1/J(r,\theta) \,dr$ is finite for (Lebesgue) almost every $\theta\in \mS^1$. This certainly means that $M$ is transient, as it must be for the DPI to be solvable, but also makes Theorem \ref{THM:RadialDPIInf} below seem somewhat more plausible.

Theorem \ref{THM:MainDPI} is flexible enough to accommodate flats, as described in Section \ref{Sect:Flats}. Moreover, curvature bounds give estimates on $J$ via the Jacobi equation (see Section \ref{Sect:RadialBounds}). This allows us to show that, perhaps surprisingly in spite of the above, any radial curvature bound that implies transience also implies solvability of the DPI (which is not true in higher dimensions, as seen from Euclidean space). We can record this, slightly informally, as follows, with a more precise statement and proof given in Section \ref{Sect:RadialBounds}.
\begin{THM}[informal version of Theorem \ref{THM:RadialDPI}]\label{THM:RadialDPIInf}
Let $M$ be a Cartan-Hadamard surface, and suppose that $M$ satisfies a transient upper curvature bound. Then the DPI on $M$ is solvable.
\end{THM}
This is nice in part because whether or not an upper curvature bound implies transience essentially reduces to a tractable ODE question, and this allows one to consider other functional forms for the curvature bound in a systematic way. We discuss this in Section \ref{Sect:Radial}.

For a transient Cartan-Hadamard surface $M$, uniformization implies that $\partial M$ is $\mS^1$, thought of as the boundary of the unit disk under a conformal map (and $\partial M$ is also the minimal Martin boundary). Even though $\partial M$ and $\SI(M)$ are both copies of $\mS^1$, it is not true in general that they are ``the same $\mS^1$,'' that is, they may not be naturally homeomorphic as discussed above. However, the potential relationships between them can be described in more detail for surfaces than in higher dimensions, and in Section \ref{Sect:Relating} we describe the various possibilities. Moreover, if the DPI is solvable, we can give the following relationship between $\SI(M)$ and $\partial M$, and a natural condition characterizing when they are homeomorphic.
\begin{THM}\label{THM:Homeo}
Let $M$ be a Cartan-Hadamard surface on which the DPI is solvable. Then there is a natural surjection of the Martin boundary onto $\SI(M)$, and this surjection is continuous. Further, this surjection is a homeomorphism if and only if the hitting measure of Brownian motion (started from any point) on $\SI(M)$ has no atoms, or equivalently, if and only if each geodesic ray (from any pole) accumulates at a single point of the Martin boundary.
\end{THM}

Starting from this condition and uniformization, we can give a natural geometric approach to proving that certain pairs of upper and lower curvature bounds imply  that $\SI(M)$ and $\partial M$ are naturally homeomorphic. In particular, we have the following, proven in Section \ref{Sect:MartinId}.

\begin{THM}\label{THM:MartinId}
Let $M$ be a Cartan-Hadamard surface, and suppose that, for some $p\in M$, we have one of the following pairs of upper and lower radial curvature bounds
\begin{itemize}
\item $K$ is bounded from below and, for some $\eps>0$ and $R>0$,
\[
K(r,\theta)\leq -\frac{2+\eps}{r^2} \quad \text{for $r>R$}
\]
or,
\item $K\leq -1$ and for some $C>0$ and $R>0$ and some $0<\lambda<2$,
\[
K(r,\theta)\geq -C e^{\lambda r} \quad \text{for $r>R$.}
\]
\end{itemize}
Then the Martin boundary of $M$ is naturally homeomorphic to $\bS_{\infty}(M)$.
\end{THM}

Note that differences in the constants in such curvature bounds generally translate into substantial differences in the asymptotic behavior of $J$ (see, for example, Section 2 of \cite{MyIll}), so depending on one's point of view, the improvement over the bounds in Section \ref{Sect:Previous} that we obtain for surfaces is considerable.

\subsection{Acknowledgements} This work was partially supported by a grant from the Simons Foundation.

\section{Background results}

From now on, we let $M$ denote a Cartan-Hadamard surface, frequently with choice of pole $p$ and polar coordinates. We begin with a more precise formulation of the basic relationship between the angular behavior of Brownian motion and the Dirichlet problem at infinity on a surface. In what follows, we let $B_t$ be Brownian motion on $M$ and $\zeta$ its lifetime (in particular, there is no need to assume $M$ is stochastically complete). Also, $\Prob^x$ and $\E^x$ refer to the probability of an event and the expectation of a random variable, respectively, with respect to Brownian motion started from $x\in M$. We will generally write $r_t$ and $\theta_t$ for $r(B_t)$ and $\theta(B_t)$, respectively.

\begin{THM}\label{THM:Basic}
Let $M$ be a Cartan-Hadamard surface. Then the following are equivalent:
\begin{enumerate}

\item Brownian motion started from some (and hence any) point in $M$ converges in the cone topology (that is, $B_{\zeta}\in \SI(M)$ exists a.s.), and for any $\hat{\theta}_0\in \SI(M)$ and any neighborhood $N\subset\SI(M)$ of $\hat{\theta}_0$,
\[
\lim_{x\rightarrow \hat{\theta}_0} \Prob^x\lp B_{\zeta}\in N\rp =1 .
\]

\item For some pole $p$ and corresponding system of polar coordinates (and hence for any $p$), for any $\theta_0$ and any $\delta>0$, there exists $R>0$ such that
\[
\Prob^{(r_0,\theta_0)}\lp \|\theta_t-\theta_0\|<\delta \text{ for all $t\in[0,\zeta)$}\rp >1-\delta 
\]
whenever $r_0>R$. (Here $\| \cdot\|$ is the standard distance on $\bS^{1}\simeq \SI(M)$.)

\item The DPI on $M$ is (uniquely) solvable.

\end{enumerate}
Given any (hence all) of these conditions, the (unique) solution to the DPI, for any continuous boundary data $g$, is given by
\begin{equation}\label{Eqn:BasicRep}
u(x) = \E^x\lb g\lp B_{\zeta}\rp\rb  \quad\text{for any $x\in M$.}
\end{equation}
\end{THM}

\begin{proof} First note that any of these three conditions clearly implies that $M$ is transient.
That (1) implies (3) is the basis for the probabilistic approach to the DPI and is well known, as is the fact that condition (1) implies the representation of the solution given in Equation \eqref{Eqn:BasicRep}-- see Proposition 6.1.1 of \cite{EltonBook}, for example. Note that $\theta_t$ converging, starting from some $x\in M$, is a tail-measurable event, and thus the probability that $\theta_t$ converges is a harmonic function of $x$, with values in $[0,1]$. Thus if $\theta_t$ converges a.s.\ starting from some $x$, by the (strong) maximum principle, $\theta_t$ converges a.s.\ starting from any $x\in M$, which also means that the displayed expression is well defined.

We next show that (3) implies (2) holds for any $p$. Choose any $p$ with some associated polar coordinates, and choose $\theta_0$ and $\delta$. Then we can take a continuous $g$ on $\SI(M)$ such that $0\leq g \leq 1$, $g= 1$ on some neighborhood of $\hat{\theta}_0$, and $g(\hat{\theta})=0$ if $\|\hat{\theta}-\hat{\theta}_0\| \geq \delta/2$. Solvability of the DPI implies that there is a unique harmonic function $f$ on $M$ with boundary values given by $g$. If we now consider $f(B_t)$ for Brownian motion started at $x$, then $f(B_t)$ is a (continuous) martingale with values in $[0,1]$ and $\E\lb f\lp B_{t\wedge\zeta}\rp\rb=f(x)$. It follows from standard estimates for martingales that $f\lp B_{t\wedge\zeta}\rp$ converges as $t\rightarrow\zeta$ almost surely and that for any $\eta>0$ there exists $\eps>0$ such that, if $f(x)>1-\eps$, then $f\lp B_{t\wedge\zeta}\rp$ stays above $1-\eta$ with probability as least $1-\eta$. Now $f$ is continuous on $M\cup \SI(M)$ and has boundary values given by $g$, and thus, for small enough $\eta$, the superlevel set $\{f>1-\eta\}$ will be contained in the cone $\{\|\theta_t-\theta_0\|<\delta, r>0\}$. Since $f(r,\theta_0)$ converges to 1 as $r\rightarrow\infty$, we see that for large enough $R$, the ray $[R,\infty)\times\{\theta_0\}$ is contained in the superlevel set $\{f>1-\eps\}$. It follows that this $R$ satisfies the condition of (2).

Finally, we show that having property (2) for some $p$ implies (1). Choose some $\delta>0$, and let $\gamma_{\delta}(x)$  be the probability that, for Brownian motion started from $x\in M$, all the accumulation points of $\theta_t$ in $\SI(M)$ (as $t\rightarrow \zeta$) lie in an interval of length $\delta$ (that is, the ``asymptotic oscillation'' of $\theta_t$ is no more than $\delta$). Now let $\hat{\theta}_1<\ldots< \hat{\theta}_n$ be a partition of $\SI(M)$ into intervals of length less than $\delta/3$ (that is, $\| \hat{\theta}_i - \hat{\theta}_{i -1}\|<\delta/3$, where $\hat{\theta}$ is understood mod $2\pi$ and the indices are understood mod $n$, as necessary). Then by (3) and the transience of $M$, for any $\eta>0$, we can find $R>0$ such that, for all $i$,
\[
\Prob^{(r_0,\theta_i)}\lp \|\theta_t-\theta_0\|<\delta/3 \text{ for all $t\in[0,\zeta)$}\rp >1-\eta/3 
\]
whenever $r_0>R/2$, and for any $\theta$,
\[
\Prob^{(r_0,\theta)} \lp r_t>R/2 \text{ for all $t\in[0,\zeta)$}\rp >1-\eta/3 
\]
whenever $r_0\geq R$. Now consider Brownian motion started from some $(r_0,\theta_0)$ with $r_0>R$. Let $i$ be such that $\hat{\theta}_{i-1}\leq \theta_0 < \hat{\theta}_i$. Then with probability $1-\eta/3$, $r_t>R/2$ for all time. In this case, one of two things happens. Either the process hits one of the rays $(R/2,\infty)\times \{\hat{\theta}_{i-1}\}$ or $(R/2,\infty)\times \{\hat{\theta}_{i}\}$, in which case we see that $\|\theta_t-\theta_0\|< 2\delta/3$ for all $t\in[0,\zeta)$ with probability $1-\eta/3$, or the process stays between these two rays for all time, in which case $\|\theta_t-\theta_0\|< \delta/3$ for all $t\in[0,\zeta)$ with probability 1. Putting this together, we see that $\gamma_{\delta}\geq 1-\eta$ on $\{r>R\}$. Since $\gamma_{\delta}$ depends only on the asymptotic behavior of $B_t$, we see that $\gamma_{\delta}$ is harmonic. Then since $\eta$ was arbitrary, it follows from the mean-value property of $\gamma_{\delta}$ that $\gamma_{\delta}\equiv 1$ on $M$. Since this is true for every $\delta>0$, recalling the definition of $\gamma_{\delta}$, we see that $B_t$ almost surely converges in the cone topology. Next, choose some $\hat{\theta}$ and neighborhood $N$ as in (2). Then we can find $\delta>0$ such that $(\hat{\theta}-\delta, \hat{\theta}+\delta)\subset N$. From the preceding argument, we know that, for any $\eta>0$, we can find $R$ such that $\|\theta_t-\theta_0\|<\delta$ with probability $1-\eta$ whenever $r_0>R$. Then it follows that 
\[
\lim_{x\rightarrow \hat{\theta}} \Prob^x\lp B_{\zeta}\in N\rp =1 ,
\]
and we've established (1).

Also note that the preceding shows that if condition (2) holds for some $p$, then it holds for every $p$.
\end{proof}

Next, we recall two ``soft'' results that will be used in what follows, and for which we need some definitions. As usual, we assume we have polar coordinates around a pole $p$. For any $\theta_0$, $\alpha>0$, and $\beta\in (0,\pi)$, we have the (open) truncated sector
\[
S_{\theta_0}\lp \alpha,\beta\rp = \lc (r,\theta): r>\alpha , \theta\in \lp \theta_0-\beta,\theta_0+\beta\rp\rc \subset M
\]
centered around $\theta_0$. Given such an $S = S_{\theta_0}\lp \alpha,\beta\rp$, an exhaustion of $S$ will mean an increasing sequence of open subsets $S_1\subset S_2\subset \cdots$ such that $\overline{S_n}$ is compact and contained in $S$ for all $n$, and such that for any compact $\Omega\subset M$, there exists $n$ such that $\Omega\subset S_n$. Next, if $\mu$ is a probability measure supported in $S$, then we let $G_\mu^S$ be the associated Green's function, with Dirichlet boundary conditions on the (finite part of) the boundary. Recall that $G^S_{\mu}$ gives the occupation density, with respect to $d(\area)$, of Brownian motion started at $\mu$ and killed at the boundary (and explosion time). Then $G^{S_n}_{\mu}$ is defined analogously for any exhaustion of $S$ and any $n$ with $\mu$ supported in $S_n$. (All of these Green's functions will be locally integrable, since $S$ is clearly conformally equivalent to the open unit disk.)

\begin{Lemma}\label{Lem:Mu}
Let $M$ be a Cartan-Hadamard surface with polar coordinates around $p$. For some point $(r',\theta_0)$, let $S=S_{\theta_0}(\alpha,\beta)$ be any truncated sector containing $(r',\theta_0)$.  Then there exists $r''>r'$ and a probability measure $\mu$ such that the support of $\mu$ is $[r',r'']\times\{\theta_0\}$ and $G^S_{\mu}$ is identically equal to one on $[r',r'']\times\{\theta_0\}$.  We have that
\[
\int_{S}\left|\nabla G^S_{\mu}\right|^2 \, d(\text{area}) = 2 ,
\]
where $\nabla G^S_{\mu}$ is interpreted in the weak sense on $\supp(\mu)$. Further, there is an exhaustion $S_n$ of $S$ such that 
\[
\int_{S_n}\left|\nabla G^{S_n}_{\mu}\right|^2 \, d(\text{area}) < 2 \quad\text{for all $n$,} 
\]
where again $\nabla G^{S_n}_{\mu}$ is interpreted in the weak sense on $\supp(\mu)$. 
\end{Lemma}

This is Lemma 5 of \cite{MyDPI}. Note that by conformal invariance, it is enough to prove that the result holds for divergent curves on the 2-dimensional hyperbolic space $\bH^2$ (rather than geodesics on $M$). Using that the Green's function on $\bH^2$ decays uniformly from any point (with respect to the hyperbolic distance), it's easy to see that such a $\mu$ exists, and then the rest of the theorem is an exercise in integration by parts. 

\begin{THM}\label{THM:Corraled}
Let $M$ be a Cartan-Hadamard surface with polar coordinates around a pole $p$, and suppose that, for any $\theta_0$, $R>0$, and $\delta>0$, 
there exists $r_0$ with $r_0>R$ such that
\[
\Prob^{(r_0,\theta_0)}\lp \|\theta_t-\theta_0\|<\delta \text{ for all $t\in[0,\zeta)$}\rp >1-\delta .
\]
Then the Dirichlet problem at infinity on $M$ is solvable. (Note that the condition is not that the inequality holds for every $r_0>R$, but rather just for some $r_0>R$, which is the key difference with condition (2) of Theorem \ref{THM:Basic}.)
\end{THM}

The proof uses the fact that curves locally divide surfaces to ``corral'' Brownian motion from all points with large $r$ by using Brownian motion from a finite number of the points assumed in the theorem. The logic is quite similar to that used to show that (3) implies (2) in the proof of Theorem \ref{THM:Basic}, and the details can be found in Lemma 9 and the subsequent proof of Theorem 7 in \cite{MyDPI}.

\section{Solvability of the DPI on a surface}\label{Sect:DPI}

Recall that in polar coordinates around a pole $p$, we write the metric on $M$ as 
\[
ds^2 = dr^2 + J^2(r,\theta) \d\theta^2
\]
Then $K(r,\theta)$ determines $J(r,\theta)$ via the Jacobi equation
\[
\partial_r^2 J(r,\theta) + K(r,\theta)J(r,\theta)=0 \quad\text{with } J(0,\theta)=0\text{ and } \partial_r J(0,\theta)=1
\]
along rays (and vice versa via $K(r,\theta)=-\frac{\partial_r^2 J}{J}(r,\theta)$). Note also that $d(\area) = J(r,\theta)\,dr\,d\theta$ and $|\nabla\theta|=\frac{1}{J(r,\theta)}$ in polar coordinates.

\subsection{A criterion for DPI in terms of $J$}\label{Sect:InJ} The proof of our basic criterion for solvability of the DPI is an improvement on the ideas in the proof of Lemma 8 in \cite{MyDPI}.

\begin{proof}[Proof of Theorem \ref{THM:MainDPI}]
First, as noted above, the hypothesis implies that $M$ is transient. Let $\theta_0$, $R$, and $\delta$ be as in Theorem \ref{THM:Corraled}; we wish to find a corresponding $r_0$. Letting $\beta=\delta$ and $\eps=\delta^4/8$, we can find corresponding $\rho$ as in the theorem, and we can also also choose $\alpha$ such that $\alpha>\rho$ and $\alpha>R$. Let $\sigma_{\alpha}$ be the first hitting time of $\{r=\alpha\}$ and $\sigma_\beta$ be the first hitting time of $\{|\theta-\theta_0|=\beta=\delta\}$. Then the first hitting time of $\partial S_{\theta_0}(\alpha,\beta)$ is $\sigma_{\alpha}\wedge\sigma_{\beta}$.  By transience, we can find $r'>\alpha$ large enough so that, for any $r_0\geq r'$,
\begin{equation}\label{Eqn:SigmaAlpha}
\Prob^{(r_0,\theta_0)}\lb \sigma_{\alpha}<\infty\rb < \frac{\delta}{2} .
\end{equation}

Next, given $S=S_{\theta_0}(\alpha,\beta)$ and $r'$ as above, let $r''$, $\mu$, and $S_n$ be the corresponding radius, measure, and exhaustion from Lemma \ref{Lem:Mu}. Let $\sigma_n$ be the first hitting time of $\partial S_n$, and note that $\sigma_n\nearrow \lp \sigma_{\alpha}\wedge\sigma_{\beta}\rp$ as $n\rightarrow \infty$. We want to bound the expectation of $(\theta-\theta_0)^2(B_{\sigma_n})$, for Brownian motion started from $\mu$. Writing $(\theta-\theta_0)^2_t$ for $(\theta-\theta_0)^2(B_{t})$, It\^o's rule shows that
\[
d\lp (\theta-\theta_0)^2 \rp_t = 2(\theta-\theta_0) \left|\nabla \theta\right|\, dW_t +\lp (\theta-\theta_0)\Lap\theta +\left|\nabla \theta\right|^2 \rp\, dt .
\]
Recall that $S_n$ is compact. Then $G^{S_n}_{\mu}$, which gives the occupation density of Brownian motion from $\mu$, stopped at $\sigma_n$, is integrable. Also, $(\theta-\theta_0)_0= 0$, so we have
\[
\E^{\mu}\lb \lp \theta-\theta_0 \rp^2_{ \sigma_n } \rb =
\int_{S_n} (\theta-\theta_0)\Lap\theta G^{S_n}_{\mu} \, d(\text{area}) 
+ \int_{S_n} \left|\nabla \theta\right|^2 G^{S_n}_{\mu} \, d(\text{area}) .
\]
Integration by parts gives
\[\begin{split}
\int_{S_n}  \lp(\theta-\theta_0) G^{S_n}_{\mu} \rp \Lap\theta \,d(\text{area}) & =
\int_{\partial S_n} (\theta-\theta_0) G_{\mu}^{S_n} \ip{\nabla\theta}{\vn} \,d(\text{length}) \\
 -\int_{S_n} G^{S_n}_{\mu} \left| \nabla\theta \right|^2  \,d(\text{area}) 
& - \int_{S_n} (\theta-\theta_0) \ip{\nabla G^{S_n}_{\mu}}{\nabla\theta}  \,d(\text{area}) ,
\end{split}\]
where $\vn$ is the outward unit normal. We know that $G^{S_n}_{\mu}$ vanishes on the boundary of $S_n$, and thus the boundary term is zero. Using this in the above and simplifying, we find
\[
\E^{\mu}\lb \lp \theta-\theta_0 \rp^2_{ \sigma_n } \rb =
- \int_{S_n} (\theta-\theta_0) \ip{\nabla G^{S_n}_{\mu}}{\nabla\theta}  \,d(\text{area}) .
\]
Now the fact that $|\theta-\theta_0|\leq \delta$ on $S_n$ plus the Cauchy-Schwarz inequality imply that
\[
\E^{\mu}\lb \lp \theta-\theta_0 \rp^2_{ \sigma_n } \rb \leq \delta 
\sqrt{\int_{S_n} \lab\nabla\theta\rab^2  \,d(\text{area})}
\sqrt{ \int_{S_n} \lab\nabla G^{S_n}_{\mu}\rab^2  \,d(\text{area})  }.
\]
The $L^2$-norm of $\nabla G^{S_n}_{\mu}$ is less than $\sqrt{2}$ by Lemma \ref{Lem:Mu}. Further, since
\[
\int_{S_n} \lab\nabla\theta\rab^2  \,d(\text{area}) \leq \int_{\theta_0-\beta}^{\theta_0+\beta}  \int_{\alpha}^{\infty} \frac{1}{J^2(r,\theta)} \,d(\text{area}),
\]
we see that the $L^2$-norm of $\nabla\theta$ is less than $\sqrt{\eps}=\delta^2/(2\sqrt{2})$, by hypothesis. Thus,
\[
\E^{\mu}\lb \lp \theta-\theta_0 \rp^2_{ \sigma_n } \rb \leq \frac{\delta^3}{2}.
\]

Note that the event $\{\sigma_{\alpha}\wedge\sigma_{\beta}<\infty\}$ is given by the (not necessarily disjoint) union $\{\sigma_{\beta}<\sigma_{\alpha}\} \cup \{\sigma_{\alpha}<\infty\}$ (and this first event requires that $\sigma_{\beta}<\infty$). By the definition of $\sigma_{\beta}$ and our choice of $\beta$, we see that
\[
\delta^2\cdot \Prob^{\mu}\lp \sigma_{\beta}<\sigma_{\alpha}\rp \leq \E^{\mu}\lb \lp \theta-\theta_0 \rp^2_{\sigma_{\alpha}\wedge\sigma_{\beta}} \Ind{\sigma_{\alpha}\wedge\sigma_{\beta} <\infty} \rb ,
\]
where the expectation is well defined because $\lp \theta-\theta_0 \rp^2_{\sigma_{\alpha}\wedge\sigma_{\beta}}$ is well defined on $\{\sigma_{\alpha}\wedge\sigma_{\beta} <\infty\}$. Moreover, $(\theta-\theta_0)^2_{\sigma_n}$ converges converges to $\lp \theta-\theta_0 \rp^2_{\sigma_{\alpha}\wedge\sigma_{\beta}}$ almost surely on the event $\{\sigma_{\alpha}\wedge\sigma_{\beta} <\infty\}$, so the dominated converge theorem combined with the above shows that
\[
\Prob^{\mu}\lp \sigma_{\beta}<\sigma_{\alpha}\rp \leq \frac{1}{\delta^2} \cdot  \frac{\delta^3}{2} =\frac{\delta}{2} .
\]
Along with \eqref{Eqn:SigmaAlpha} (recall that $\mu$ is supported on $[r',r'']\times\theta_0$), this shows that
\[
\Prob^{\mu}\lp \sigma_{\beta}\wedge \sigma_{\alpha} <\infty \rp > 1-\delta .
\]
In the context of Theorem \ref{THM:Corraled}, note that the event $\{\|\theta_t-\theta_0\|<\delta \text{ for all $t\in[0,\zeta)$}\}$ contains the event  $\{\sigma_{\beta}\wedge \sigma_{\alpha} <\infty \}$. This means that we have established that the hypothesis of Theorem \ref{THM:Corraled} holds, except that our process starts from $\mu$ rather than from some $(r_0,\theta_0)$ with $r_0>R$.

However, the $\mu$-probability of any event is the $\mu$-average of the probability when the process is started from any $(r_0,\theta_0)$ in the support of $\mu$. And since $r'>R$, it follows that there is a point $(r_0,\theta_0)$ in the support of $\mu$ (indeed, ``many'' such points, in some sense) satisfying the hypothesis of Theorem \ref{THM:Corraled}, and thus the DPI on $M$ is solvable.

\end{proof}

\subsection{Flats}\label{Sect:Flats}
While our primary interest is in radial curvature bounds, for which $\int_{\alpha}^{\infty} J(r,\theta_0)\, dr $ can be made arbitrarily small by making $\alpha$ large, uniformly in $\theta_0$, we note that Theorem \ref{THM:MainDPI} naturally accommodates ``flats.'' A flat is a subset of $M$ which is isometric to a strip $[-a,a]\times \bR \subset \bR^2$ for some $a>0$ (with the flat metric). Indeed, if $p$ is in the interior of such a flat, there will be a geodesic through $p$, which, without loss of generality, we take to be given by the union of the rays $\{\theta=\pi/2\}$ and $\{\theta=3\pi/2\}$, such that $J(r,\pi/2)=J(r,3\pi/2)= r$. Then 
\[
\int_{\alpha}^{\infty} \frac{1}{J\lp r,\frac{\pi}{2}\rp} \,dr = \int_{\alpha}^{\infty} \frac{1}{J\lp r,\frac{3\pi}{2}\rp} \,dr =\infty
\]
for any $\alpha>0$, so (in contrast to the proof of Theorem \ref{THM:RadialDPI} below) we won't be able to bound the inner integral (the one with respect to $dr$) in Equation \eqref{Eqn:MainHypo} uniformly in $\theta$. Nonetheless, every ray for $\theta\not\in\{\pi/2, 3\pi/2\}$ exits this flat after a finite distance (and never returns), and thus if the curvature off of the flat is sufficiently negative, we could have that
\[
\int_{\pi/2-\beta}^{\pi/2+\beta} \lp \int_{\alpha}^{\infty} \frac{1}{J(r,\theta)} \,dr\rp \, d\theta <\infty
\]
(essentially because, viewing the inner integral as a function of $\theta$, the singularity at $\pi/2$ is nonetheless integrable), and similarly for a truncated sector around $3\pi/2$. And this allows the possibility that the hypothesis of Theorem \ref{THM:MainDPI} is satisfied in spite of the flat. 

\subsection{Radial curvature bounds}\label{Sect:RadialBounds}
We now turn our attention to radial curvature bounds. If (for some pole $p$) $K$, or equivalently $J$, does not depend on $\theta$, we say that the surface is radially symmetric (with respect to $p$), and we write $K$ and $J$ as functions of $r$ alone. We suppress the dependence on $p$ and denote such a manifold by $M^*$ with associated functions $K^*(r)$ and $J^*(r)$ (such a manifold is determined up to isometry by $K^*(r)$ anyway). Such radially symmetric manifolds serve as natural comparison manifolds. In particular, suppose that for some $M$ with some pole $p$, we have a radially symmetric surface $M^*$ such that $0\geq K^*(r)\geq K(r,\theta)$ for all $r\in[0,\infty)$ and all $\theta\in \bS^1$. Then
\[
r\leq J^*(r)\leq J(r,\theta)\quad\text{for all $r\in[0,\infty)$ and $\theta\in\bS^1$}
\]
and
\[
(\partial_r J^*/J^*)(r) \leq (\partial_r J/J)(r,\theta)\quad\text{for all $r\in[0,\infty)$ and $\theta\in\bS^1$}
\]
by standard comparison theorems. Note that this last inequality is an inequality for the drift terms in the SDE for the radial process, which is one way to show that transience of $M^*$ implies transience of $M$. Also, recall the classical result of Milnor \cite{Milnor} that a radially symmetric Cartan-Hadamard surface $M^*$ is transient if and only if
\[
\int_{\alpha}^{\infty} \frac{1}{J^*(r)} \, dr <\infty \quad \text{for some (hence any) $\alpha>0$.} 
\]
(Note this is just a radially symmetric version of Doyle's later result.)

With this background, the desired radially curvature result is an immediate corollary of Theorem \ref{THM:MainDPI}.

\begin{THM}\label{THM:RadialDPI}
Let $M$ be a Cartan-Hadamard surface, and suppose for some $p\in M$, there is a transient radially symmetric surface $M^*$ with
\[
0\geq K^*(r)\geq K(r,\theta) \quad\text{for all $r\in[0,\infty)$ and $\theta\in\bS^1$}
\]
(this is what was meant by saying $M$ satisfies a ``transient upper curvature bound'' in Theorem \ref{THM:RadialDPIInf}). Then the DPI on $M$ is solvable.
\end{THM}

\begin{proof}
Using the comparison results just mentioned, plus Milnor's characterization of transience, we have that, for any $\theta\in \mS^1$,
\[
\int_{\alpha}^{\infty} \frac{1}{J(r,\theta)} \, dr <
\int_{\alpha}^{\infty} \frac{1}{J^*(r)} \, dr <\infty \quad \text{for any $\alpha>0$.} 
\]
This second integral can be made arbitrarily small by taking $\alpha$ large enough. Then, referring to Theorem \ref{THM:MainDPI}, we see that
\[
\int_{\theta_0-\beta}^{\theta_0+\beta} \lp \int_{\alpha}^{\infty} \frac{1}{J(r,\theta)} \,dr\rp \, d\theta
\leq 2\beta \int_{\alpha}^{\infty} \frac{1}{J^*(r)} \, dr ,
\]
and, for any $\eps>0$, we can find $\rho$ such that the integral on the right-hand side is less than $\eps/(2\beta)$ whenever $\alpha>\rho$, independent of $\theta_0$. Thus $M$ satisfies the assumptions of Theorem \ref{THM:MainDPI}, and the DPI on $M$ is solvable.

\end{proof}

This is the best possible result in terms of radial curvature bounds, since any weaker radial curvature bound would be one that does not imply transience, and transience is a necessary condition for the DPI to be solvable (indeed, for there to exist any non-constant positive harmonic functions at all). 

\subsection{Radial comparison}\label{Sect:Radial}
Theorem \ref{THM:RadialDPI} (plus Milnor's criterion for transience) allows one to consider any potential radial curvature bound just in terms of ODEs (see Section 2 of \cite{MyIll} for several related examples of comparison geometry, including more detail on the computations that follow). For example, we can recover the result of \cite{MyDPI}, that the DPI is solvable on a Cartan-Hadamard surface $M$ satisfying the radial curvature bound
\begin{equation}\label{Eqn:DPIBound}
K(r,\theta) \leq -\frac{1+\eps}{r^2 \log(r)} \quad\text{for $r>R$}
\end{equation}
for some $\eps>0$ and some $R>1$.

To do so, note that if $M$ satisfies such a bound, then we can consider a (Cartan-Hadamard) radially symmetric comparison manifold $M^*$ such that
\[
K^*(r) = -\frac{1+\frac{\eps}{2}}{r^2\log r}\lp 1+\frac{\eps}{\log r}\rp \quad \text{for $r\geq A$,}
\]
for some $A>R$. The point is that for such $K^*$, we can solve the Jacobi equation more or less explicitly. In particular, we find (by first finding one solution $J_1$ and then using reduction of order to find the other independent solution) that a basis for the solution to the Jacobi equation for $K^*$ on the interval $r\in[A,\infty)$ is given by
\[
J_1(r) = r\lp \log r\rp^{1+\frac{\eps}{2}} \quad\text{and}\quad
J_2(r) = J_1(r)  \int^r \frac{1}{s^2\lp\log s \rp^{2+\eps}} \, ds . 
\]
The exact linear combination that gives $J^*(r)$ on $r\in[A,\infty)$ depends on the initial conditions at $R=A$, which in turn depend on $K^*$ and thus $K$ for $r\in[0,A]$, and which is not given. Still, we know that $J^*(r)\geq r$ for all $r>0$ by comparison with (flat) $\bR^2$, and because $ \int^r 1/(s^2\lp\log s \rp^{2+2\eps}) \, ds$ is increasing and bounded, we see that there is some constant $c>0$ such that $J^*(r) \sim c J_1(r)$ as $r\rightarrow \infty$, where ``$\sim$'' means that the ratio of the two sides approaches 1. Since $1/J_1$ can be explicitly integrated, we see from Milnor's criterion that $M^*$ is transient, and thus the DPI on $M$ is solvable, as desired.

While the curvature bound of Inequality \eqref{Eqn:DPIBound} is known to be sharp for bounds of the form $\frac{c}{r^2\log(r)}$ (where $c$ is some positive constant), Theorem \ref{THM:RadialDPI} opens up the possibility of considering more baroque functional forms (iterated logarithms, for example), by following a similar ODE-based analysis of a comparison manifold.

\section{DPI and the Martin boundary}\label{Sect:Relating}

We have already noted that, for a Cartan-Hadamard surface $M$, the geometric circle at infinity $\SI(M)$ can be given an explicit coordinate $\hat{\theta}\in[0,2\pi)$, understood in the usual way for $\bS^1$. Indeed, for any $p\in M$, polar coordinates based at $p$ determine such a coordinate uniquely up to rotation. (Assuming that we parametrize $\SI(M)$ counter-clockwise, but since any Cartan-Hadamard surface is orientable, there is no problem with assuming we have an orientation and choosing polar coordinates in a compatible way.) For the rest of this section, we assume $M$ is a transient Cartan-Hadamard surface. Then, in a similar way, we can put a coordinate on the Martin boundary. Namely, for any $p\in M$, there is a conformal diffeomorphism from $M$ to the open unit disk $D$ that takes $p$ to the origin, and this map is unique up to rotation of $D$ (see Section \ref{Sect:MartinId} for a construction of this map using the Green's function). If $\rho$ and $\varphi$ are polar coordinates on $D$, then they induce global coordinates on $M$, and the Martin boundary $\partial M$ is identified with the boundary circle $\{\rho=1\}$. Thus, $\varphi$ induces a coordinate on the Martin boundary, which we denote $\hat{\varphi}$. So for any $p\in M$, we have a coordinate $\hat{\theta}$ on $\SI(M)$ and a coordinate $\hat{\varphi}$ on $\partial M\simeq \bS^1$, each uniquely determined up to rotation. (Note that choosing a different pole $p$ changes $\varphi$ by a homeomorphism, which is easily verified since the conformal automorphisms of the disk are well known. So all of our considerations will be seen to be independent of the choice of pole, as they must be.)

In light of the above, it is natural to consider the geodesic ray $\gamma_{\theta}$ started from $p$ with initial direction $\theta$ (that is, $\gamma_{\theta}$ is the ray $[0,\infty)\times\theta$ in polar coordinates), for each $\theta\in \mS^1$. Moreover, we identify $\gamma_{\theta}$ with its image in the unit disk under the above conformal diffeomorphism, and this is the framework in which Figure \ref{Fig:Martin} is to be understood. Note that the set of such $\gamma_{\theta}$ is canonically identified with $\SI(M)$.

For an arbitrary transient Cartan-Hadamard surface, there is no reason that there should be any particular relationship between the geometric and Martin boundaries. Presumably, one could have a situation like Figure \ref{Fig:squiggle} in which an arc of geodesics accumulates at an arc of the Martin boundary, and thus neither boundary maps into the other. However, the author is not aware of a construction of such a metric. Indeed, it appears that the only ``trivial'' constructions of (transient) Cartan-Hadamard surfaces where the geometric and Martin boundaries can be see not to be naturally homeomorphic come from having flat regions and for which the geometric boundary naturally surjects onto the Martin boundary. For example, let $M$ be a Cartan-Hadamard manifold such that, in polar coordinates, the metric satisfies $K(r,\theta)=0$ on $U=\{(r,\theta) : r\geq1, \pi/3\leq\theta\leq 2\pi/3\}$ and $K(r,\theta)=-1$ on $V=\{(r,\theta) : r\geq1, 4\pi/3\leq\theta\leq 5\pi/3\}$ (so that $U$ is isometric to a truncated sector of Euclidean space and $V$ is isometric to a truncated sector of hyperbolic space). Then $M$ is transient, which follows from the fact that Brownian motion in $V$ has a positive probability of going to infinity without leaving $V$, and indeed, $\hat{\theta}\in (4\pi/3, 5\pi/3)$ will give an arc in the Martin boundary. However, Brownian motion started in $U$ almost surely cannot go to infinity without leaving $U$, so the arc $\hat{\theta}\in (\pi/3, 2\pi/3)$ collapses to a single point in the Martin boundary, as illustrated in Figure \ref{Fig:focusing}. (This also underscores the point that Theorem \ref{THM:RadialDPI} does not say that the DPI is solvable on any transient Cartan-Hadamard surface.) In this case, the map that sends $\hat{\theta}$ to $\hat{\varphi}$ by sending $\gamma_{\hat{\theta}}$ to its limit point in $\partial D$ gives the surjection of $\SI(M)$ onto $\partial M$.

\begin{figure}
  \begin{subfigure}[b]{0.4\textwidth}
    \includegraphics[width=\textwidth]{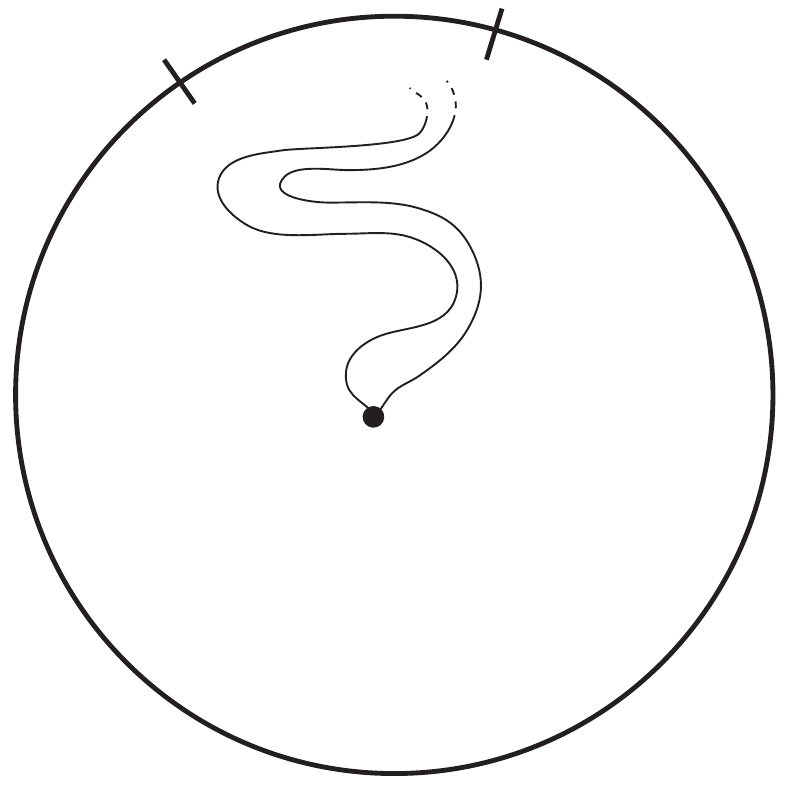}
   \caption{An arc's worth of geodesics accumulate at an arc of the Martin boundary.}
   \label{Fig:squiggle}
  \end{subfigure}
\qquad
\begin{subfigure}[b]{0.4\textwidth}
    \includegraphics[width=\textwidth]{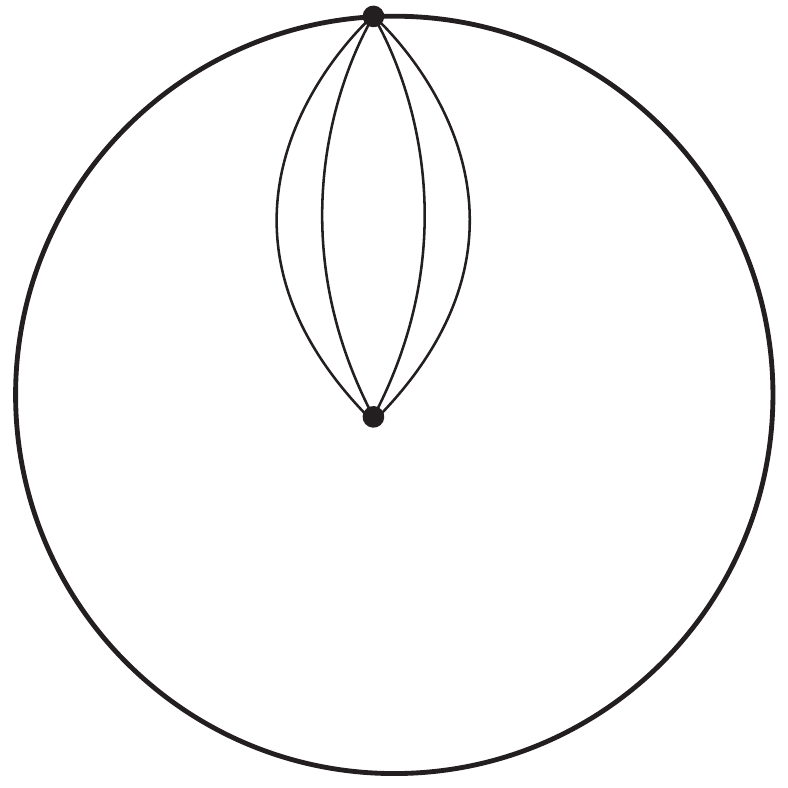}
    \caption{An arc's worth of geodesics that all converge to the same point of the Martin boundary.}
    \label{Fig:focusing}
  \end{subfigure}
 
 \bigskip
  \begin{subfigure}[b]{0.4\textwidth}
    \includegraphics[width=\textwidth]{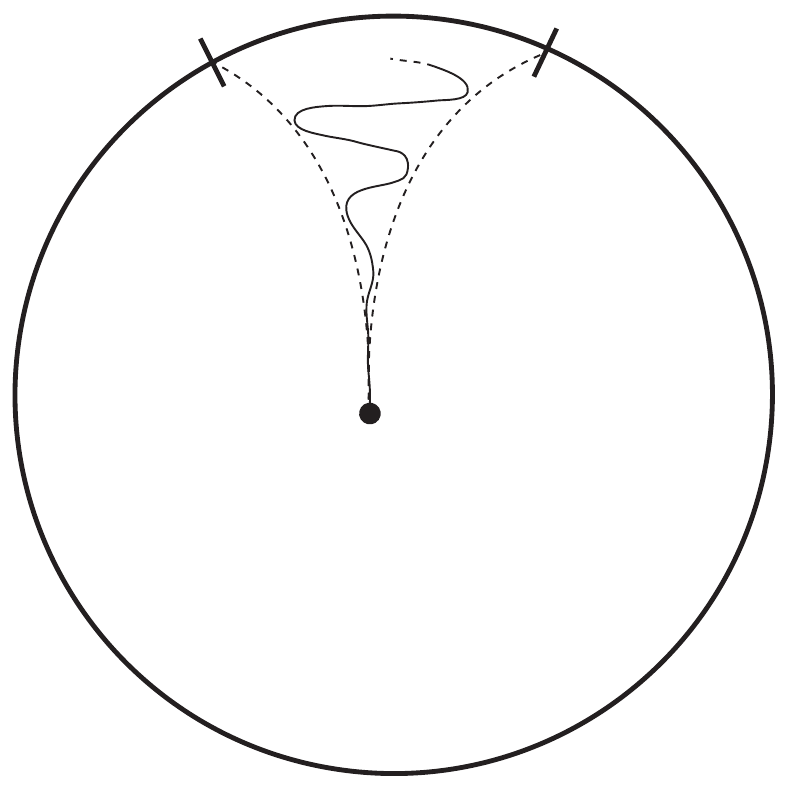}
    \caption{A single geodesic accumulates at an arc of the Marin boundary, and it is the only geodesic accumulating anywhere on this arc.}
    \label{Fig:winding}
  \end{subfigure}  
 \caption{Three (potentially) possible ways in which the geometric boundary at infinity and the Martin boundary can fail to be homeomorphic. If the the DPI is solvable, only \ref{Fig:winding} is possible.}
 \label{Fig:Martin}
\end{figure}

However, as stated in Theorem \ref{THM:Homeo}, if the DPI is solvable on a Cartan-Hadamard surface, then there is a surjection of the Martin boundary onto the sphere at infinity, which we now prove.

\begin{proof}[Proof of Theorem \ref{THM:Homeo}]
It's clear that $M$ must be transient, so consider the Martin compactification, explicitly realized as the closed unit disk under uniformization, as above, with angular coordinate $\varphi$. The central point is that for each point of the Martin boundary, meaning for each $\hat{\varphi}\in\partial D$, there is a unique $\theta$ such that the geodesic ray $\gamma_{\theta}$ accumulates at $\hat{\varphi}$. To see this, note that, if not, we can find $\varphi_0$, $\theta_0$, and $\delta>0$ such that every ray $\gamma_{\theta}$ in the arc $(\theta_0-\delta,\theta_0+\delta)$ has $\varphi_0$ as an accumulation point. But Brownian motion on $M$ is given by a (non-degenerate) time-change of Brownian motion on $D$ (under the conformal diffeomorphism relating the two), and Brownian motion on $D$, from any point, has a hitting measure on $\partial D$ that is absolutely continuous with respect to Lebesgue measure on $\mS^1$. Since both $\gamma_{\theta}$ and the paths of Brownian motion are continuous curves, it follows that, for any $r_0>0$,
\[
\Prob^{(r_0,\theta_0)}\lp \limsup_{t\rightarrow \zeta} \|\theta_t-\theta_0\|\geq \delta \rp =1 .
\]
But this contradicts Theorem \ref{THM:Basic} (indeed, this is the only part of the paper where we directly use Theorem \ref{THM:Basic}). So we conclude that for each $\hat{\varphi}\in\partial D$, there is a unique $\theta$ such that the geodesic ray $\gamma_{\theta}$ accumulates at $\hat{\varphi}$. Then let $F:\partial M\rightarrow \SI(M)$ be the map given in coordinates by taking $\hat{\varphi}\in\partial M$ to the unique $\hat{\theta}$ such that $\gamma_{\hat{\theta}}$ accumulates at $\hat{\varphi}$.

By compactness, every ray $\gamma_{\theta}$ has at least one accumulation point on $\partial D$, and thus $F$ must be surjective. Suppose that $F$ is not injective. Then if $F(\hat{\varphi}_0) =F(\hat{\varphi}_1)$ for $\hat{\varphi}_0\neq \hat{\varphi}_1$, because the rays $\gamma_{\theta}$ cannot cross outside of the origin, $F$ must be constant  on one of the two arcs between $\hat{\varphi}_0$ and $\hat{\varphi}_1$. It follows that there are at most countable many such arcs (on which $F$ is constant), and so, after rotation in $\hat{\theta}$, we can assume that $F(0)=0$ and that $0$ is the unique pre-image of $0$.

Hence we can write $F$ in coordinates as a function from $\hat{\varphi}\in [0,2\pi]$ to $\hat{\theta}\in[0,2\pi]$ such that $F(0)=0$ and $F(2\pi)=2\pi$. Moreover, because the rays $\gamma_{\theta}$ cannot cross outside of the origin, $F$ is monotone non-decreasing. Then because $F$ is a surjection, it is necessarily continuous, and $F$ is a homeomorphism if and only if $F$ is injective, which is equivalent to $F$ being strictly increasing. And $F$ only fails to be strictly increasing if there is a non-trivial interval on which it is constant, but in light of the above, this is exactly the case when there is a geodesic ray that accumulates at more than one point of $\partial D$ (which we identify with $\partial M$), in which case it must accumulate at an entire arc. And if a geodesic ray accumulates at a non-trivial arc, it must be the only geodesic that accumulates at any point of that arc (see Figure \ref{Fig:winding}). Finally, since the hitting measure of Brownian motion from any point in (the interior of) $D$ has a smooth density, bounded from below, on $\partial D$, we see that when this hitting measure is pushed to $\SI(M)$ by $F$, there is an atom if and only if the corresponding geodesic accumulates at a (non-trivial) arc.

\end{proof}

To complete the discussion of the possible relationships between the geometric boundary at infinity and the Martin boundary, note that the case when $\SI(M)$ and $\partial M$ are naturally homeomorphic is most clearly illustrated by the hyperbolic plane. Indeed, in this case, the conformal diffeomorphism of $M$ onto $D$ just gives the Poincare disk model of hyperbolic geometry and the rays $\gamma_{\theta}$ are exactly the Euclidean rays from the origin, so that the homeomorphism from $\SI(M)$ to $\partial M$ is just the identity map on $\mS^1$, after possible rotation (and it seems unnecessary to draw the corresponding picture). More generally, when $\SI(M)$ and $\partial M$ are naturally homeomorphic, the homeomorphism is given by the map $F$ above that takes $\hat{\theta}$ to the unique $\hat{\varphi}$ such that $\gamma_{\hat{\theta}}$ accumulates at $\hat{\varphi}$.

\section{Identification of the Martin boundary with $\SI(M)$}\label{Sect:MartinId}

Again assume $M$ is a transient Cartan-Hadamard surface, and let $G^p(r,\theta)$ be the Green's function associated to a point mass at $p$, that is, the occupation density of Brownian motion started from $p$, where $p$ is the pole for our polar coordinates. Then it is well known that we can can give a conformal map from $M$ to the unit disk $D$ using $G^p$. In particular, $2\pi G^p$ is harmonic on $M\setminus\{p\}$ and has distributional Laplacian $-2\pi\delta_{p}$. Then if $F^p$ is the harmonic conjugate of $2\pi G^p$ on $m\setminus\{p\}$, $F^p$ is only defined modulo $2\pi$ and up to an additive constant. Nonetheless, if we fix the additive constant, so that $F^p$ is a multi-valued function, then $g=e^{-\lp 2\pi G^p + iF^p\rp}$ is a well-defined (that is, single-valued) conformal bijection from $M$ to $D$, here viewed as a subset of $\mC$. Moreover, $g$ maps $p$ to the origin and we have specified $g$ up to the additive constant in $F^p$, which corresponds to specifying it up to a rotation of the disk. (Indeed, one can compare this construction to the explicit Green's function on the disk to see the motivation.)

From the above, we see that the conformal factor relating the metric in $M$ to the Euclidean metric on $D$ is $2\pi\|\nabla G^p\| |g|$. As above, let $\gamma_{\theta_0}$ be the image in $D$ of the geodesic ray from $p$ corresponding to $\{\theta=\theta_0\}$ in polar coordinates. Then the Euclidean length of $\gamma_{\theta_0}$ in $D$ is given by integrating this conformal factor along the geodesic ray. In particular, since $\gamma_{\theta}$ is a proper curve (so that the length of the portion in any neighborhood of the origin is finite) and the conformal factor is comparable to $\|\nabla G^p\|$ outside of any neighborhood of the origin, $\gamma_{\theta_0}$ will have finite Euclidean length if and only if 
\[
\int^{\infty} \|\nabla G^p(r,\theta_0)\| \, dr <\infty .
\]
Further, in light of Section \ref{Sect:Relating}, we see that $\gamma_{\theta_0}$ having finite Euclidean length implies that the hitting measure on $\SI(M)$ cannot charge $\theta_0$. Thus, to prove that $\partial M$ and $\SI(M)$ are naturally homeomorphic, it is sufficient to show that $\gamma_{\theta_0}$ has finite Euclidean length for every $\theta_0$, and thus it is sufficient to show that $\|\nabla G^p(r,\theta_0)\|$ is integrable near $r=\infty$ for every (fixed) $\theta_0$.

This will be our approach to proving Theorem \ref{THM:MartinId}. To carry this out, we will make use of the Cheng-Yau inequality (which goes back to \cite{ChengYau}, and we also refer to \cite{ThalADriver} for a version proven by stochastic methods and also proving an explicit value for the dimensional constant, though we won't use that explicit value). The Cheng-Yau inequality compares the gradient of a positive harmonic function to the value of the function at every point of a relatively compact domain. For any $r_0>3$, we consider the annulus $A_{r_0}=\lc r_0-1<r<r_0+1, \theta\in \mS^1\rc$, and then we let $k(r_0)=\sup_{A_{r_0}} \sqrt{-K(r,\theta)}$. Then since $G^p$ is positive and harmonic away from $p$, for any $r>3$, the Cheng-Yau inequality implies the estimate
\begin{equation}\label{Eqn:CY}
\| \nabla G^p(r,\theta)\| \leq c \lp k(r) +1\rp G^p(r,\theta) ,
\end{equation}
where $c$ is a dimensional constant.

Thus, given a radial lower curvature bound, we will be able to estimate $\| \nabla G^p\|$ in terms of $G^p$. Then to estimate $G^p$, we will use a radial upper curvature bound in the following way. First, because $G^p$ is a positive, bounded harmonic function on $\{r>2\}$ that goes to 0 as $r\rightarrow\infty$, $G^p$ on $\{r>2\}$ is determined by its boundary values on $\{r=2\}$, and these boundary values are continuous and bounded above and below by some positive constants. More precisely, $G^p(r_0,\theta_0)$ for $r>2$ is given as the integral of these boundary values with respect to the hitting measure on $\{r=2\}$ of Brownian motion started at $(r_0,\theta_0)$, where this hitting measure will have total mass strictly less than 1 by the transience of $M$ (if we think about the unit disk under uniformization, we're just solving the Dirichlet problem on a topological annulus where the outer boundary has 0 boundary value). Now if $M^*$ is a (transient) radially symmetric upper comparison manifold as in Section \ref{Sect:RadialBounds} with Green's function at its pole written as $G^*(r)$, we have an analogous representation of $G^*(r)$ as the mass of the hitting measure from $r$ times $G^*(2)$ (taking advantage of the radial symmetry to reduce everything to 1 dimension). Radial comparison implies that the probability that Brownian motion on $M$ from $(r_0,\theta_0)$ hits $\{r=2\}$ is less than or equal to the probability that Brownian motion on $M^*$ from $(r_0,\cdot)$ hits $\{r=2\}$. Putting this together, we see that, for any $r_0>2$ and any $\theta_0$, $G^p(r_0,\theta_0)$ is less or equal to a constant times $G^*(r_0)$, where this constant depends only on the ratio of $\max_{\theta} G^p(2,\theta)$ and $G^*(2)$, and thus not on $r_0$ or $\theta_0$. Further, on a radially symmetric manifold, $G^*(r)$ can be readily expressed in terms of $J^*(r)$. Namely, one checks that, up to a normalizing constant depending only on dimension, $G^*(r)$ is given by $\int_r^{\infty} \frac{1}{J^*(s)}\, ds$. (Note that in 2 dimensions, the finiteness of this integral is exactly the well-known condition for $M^*$ to be transient.) Thus, using standard comparison geometry as in Section \ref{Sect:Radial}, upper curvature bounds will translate into upper bounds on $G^p$ for large enough $r$.

\begin{proof}[Proof of Theorem \ref{THM:MartinId}]
First note that, in both cases, the DPI on $M$ is solvable by Theorem \ref{THM:RadialDPI} (in particular, $M$ is transient, and also Theorem \ref{THM:Homeo} applies). Referring to the discussion above, for the first set of conditions, we have that $k$ is constant in the Inequality \eqref{Eqn:CY}, so $|\nabla G^p|$ will be bounded by a constant times $G^p$ for large $r$, and thus it is sufficient to find a comparison manifold such that $G^*(r)$ is integrable near infinity. The argument is analogous to that of Section \ref{Sect:Radial}. In particular, let $\delta>0$ be such that $2+\eps=(2+\delta)(1+\delta)$. Then we can find a (radially symmetric Cartan-Hadamard) comparison manifold $M^*$ such that
\[
K^*(r) = -\frac{(2+\delta)(1+\delta)}{r^2} \quad \text{for $r>R$} ,
\]
so that 
\[
J_1(r)= r^{2+\delta}  \quad\text{and}\quad J_2(r)= J_1(r)\int^r \frac{1}{s^{4+2\delta}}\, ds
\]
are a basis for the solution space of the Jacobi equation on $r\in(R,\infty)$. Then there is some $c>0$ such that
\[
J^*(r) \sim c r^{2+\delta} \quad\text{as $r\rightarrow\infty$.}
\]
From here, it's straightforward to see that
\[
\int^{\infty} \lp \int_{r}^{\infty} \frac{1}{J^*(\rho)}\, d\rho\rp \, dr <\infty ,
\]
which is equivalent to showing that $G^*(r)$ is integrable near infinity.

For the second set of conditions, again referring to the discussion above, we have $k(r)=\sqrt{c e^{\lambda(r+1)}}$ in Inequality \eqref{Eqn:CY}. Thus, after adjusting $c$, we have that $|\nabla G^p| \leq c e^{\lambda r/2} G^*(r)$ for all sufficiently large $r$, independent of $\theta$. On the other hand, the upper curvature bound $K\leq -1$ means that we can take the hyperbolic plane as our comparison manifold. Here the Green's function is well known, but for our purposes, we note that, up to a constant, it is given by
\[
\int_r^{\infty} \frac{1}{\sinh \rho} \,d\rho = -\log\lp \tanh(r/2)\rp .
\]
Thus it's enough to show that
\[
\int^{\infty} - e^{\lambda r/2} \log\lp \tanh(r/2)\rp \, dr <\infty .
\]
But, restricting our attention to positive $\lambda$, we see that this is true exactly when $\lambda <2$, say, by using that $\log\lp \tanh x\rp \sim -2/(1+e^{2x})$ for large $x$.
\end{proof}

Presumably other matched pairs of upper and lower curvature bounds can be produced by the same approach, again as in Hsu's work.

Of course, in the above, we are not necessarily giving sharp results. In particular, it is not necessary for $\gamma_{\theta_0}$ to have finite Euclidean length in order for it to converge at the boundary of the disk. One just needs that the oscillations in $\varphi$ cancel sufficiently for $\varphi$ to converge along $\gamma_{\theta_0}$. So one might hope that better curvature bounds are possible. Moreover, in light of the results for the DPI, one might naturally wonder if lower curvature bounds are necessary for the identification of the geometric boundary at infinity with the Martin boundary. While we don't resolve that here, it's not clear that it is a reasonable hope. As Section \ref{Sect:Relating} illustrates, the natural inclusion of the geometric boundary at infinity into the Martin boundary and the natural inclusion of the Martin boundary into the geometric boundary at infinity are in some sense ``opposite phenomenon.'' In one case, one wants to prevent more than one geodesic from accumulating at a given point of the Martin boundary, and in the other one wants to prevent a geodesic from accumulating at multiple Martin boundary points.

Finally, we note that one can give a more stochastic argument for the results of Theorem \ref{THM:MartinId} that, in the end, amounts to the same thing. Namely, for any given $\theta_0$, let $f_{\lambda}(\theta)$ be the continuous function on $\mS^1$ that is 0 on the complement of $(\theta_0-\lambda,\theta_0+\lambda)$, 1 when $\theta=\theta_0$, and linear on $[\theta_0-\lambda,\theta_0]$ and also on $[\theta_0,\theta_0+\lambda]$, for any $\lambda\in(0,1]$ (that is, $f_{\lambda}$ is a triangle function that gets narrower as $\lambda$ decreases). Then we can view $f_{\lambda}$ as a function on $\lp M\cup \SI(M)\rp\setminus \{p\}$ in the natural way using the polar coordinates. Letting $\mu$ be a conveniently chosen initial probability measure with a smooth, compactly supported density, we see that $\lim_{\lambda\searrow 0} \E^{\mu}\lb f_{\lambda}\lp B_{\zeta}\rp\rb$ gives the mass at $\theta_0\in\SI(M)$ of the hitting measure corresponding to $\mu$, and so we want conditions under which this limit is 0. Starting from the SDE satisfied by $f_{\lambda}$ and using the $G^{\mu}$ as the occupation density of Brownian motion, integration by parts, and standard approximations and exhaustions, similar to the what was done in Section \ref{Sect:InJ}, we see that we want to show that, for large $R$,
\[
\frac{1}{\lambda} \int_{R}^{\infty}\lb \int_{-\lambda}^{0} \frac{1}{J}\frac{\partial G^{\mu}}{\partial\theta} \,d\theta -
\int_0^{\lambda} \frac{1}{J}\frac{\partial G^{\mu}}{\partial\theta} \,d\theta\rb\,dr 
\]
goes to 0 with $\lambda$ (assuming we can also justify exchanging limits with integration). Then the estimates on $\|\nabla G^p\|$ just given can also be used to show that this limit is 0 under the given curvature conditions, and they also justify the various limits. (Note that $(1/J) \partial G/\partial\theta$ is the $\theta$-component of $\nabla G$.) The previous argument in terms of the Euclidean length of $\gamma_{\theta_0}$ seems technically simpler and more intuitive, which is why we prefer it. It is also the case that we see the same possibilities for improvement, namely, it would be enough to control the $\theta$-component of the gradient, which corresponds to controlling the $\varphi$ oscillation along $\gamma_{\theta_0}$, and one could take advantage of the cancellation of these two $d\theta$ integrals, which corresponds to allowing $\gamma_{\theta_0}$ to have infinite oscillation in $\varphi$ but nonetheless having it converge.

\def\cprime{$'$}
\providecommand{\bysame}{\leavevmode\hbox to3em{\hrulefill}\thinspace}
\providecommand{\MR}{\relax\ifhmode\unskip\space\fi MR }
\providecommand{\MRhref}[2]{%
  \href{http://www.ams.org/mathscinet-getitem?mr=#1}{#2}
}
\providecommand{\href}[2]{#2}

\end{document}